\documentclass{amsart}

\usepackage[latin1]{inputenc}
\usepackage{amsfonts}
\usepackage{amsmath}
\usepackage{amsthm}
\usepackage{amssymb}
\usepackage{latexsym}
\usepackage{enumerate}

\newtheorem{theorem}{Theorem}[section]
\newtheorem{lemma}[theorem]{Lemma}
\newtheorem{proposition}[theorem]{Proposition}

\newtheorem{definition}[theorem]{Definition}

\numberwithin{equation}{section}

\begin{document}

\newcommand{\cc}{\mathfrak{c}}
\newcommand{\N}{\mathbb{N}}
\newcommand{\Q}{\mathbb{Q}}
\newcommand{\C}{\mathbb{C}}
\newcommand{\Z}{\mathbb{Z}}
\newcommand{\R}{\mathbb{R}}
\newcommand{\T}{\mathbb{T}}
\newcommand{\I}{\mathcal{I}}
\newcommand{\CC}{\mathcal{C}}
\newcommand{\J}{\mathcal{J}}
\newcommand{\F}{\mathcal{F}}
\newcommand{\G}{\mathcal{G}}
\newcommand{\A}{\mathcal{A}}
\newcommand{\HH}{\mathcal{H}}
\newcommand{\K}{\mathcal{K}}
\newcommand{\B}{\mathcal{B}}
\newcommand{\st}{*}
\newcommand{\PP}{\mathbb{P}}
\newcommand{\SSS}{\mathbb{S}}
\newcommand{\forces}{\Vdash}
\newcommand{\dom}{\text{dom}}
\newcommand{\osc}{\text{osc}}
\newcommand\encircle[1]{%
  \tikz[baseline=(X.base)] 
    \node (X) [draw, shape=circle, inner sep=0] {\strut #1};}

\title[A non-stable $C^*$-algebra]{A non-stable $C^*$-algebra
with an elementary essential 
 composition series}

\address{Institute of Mathematics, Polish Academy of Sciences,
ul. \'Sniadeckich 8,  00-656 Warszawa, Poland}

\author{Saeed Ghasemi}
\curraddr{Institute of Mathematics of the Czech Academy of Sciences, \v Zitn\'a
25, 115 67 Praha 1, Czech Republic }
\email{\texttt{ghasemi@math.cas.cz}}

\author{Piotr Koszmider}
\address{Institute of Mathematics, Polish Academy of Sciences,
ul. \'Sniadeckich 8,  00-656 Warszawa, Poland}
\email{\texttt{piotr.koszmider@impan.pl}}
\thanks{The research of the second named author was partially supported by   grant
PVE Ci\^encia sem Fronteiras - CNPq (406239/2013-4).}

%\subjclass[2010]{54D30,  03E35, 46B25}
\begin{abstract} 
A $C^*$-algebra $\A$ is said to be stable if it is isomorphic to $\A \otimes \K(\ell_2)$.
Hjelmborg and R\o rdam have shown that  countable inductive limits of separable stable $C^*$-algebras are stable.
We  show that this is no longer true in the nonseparable
context even for the most natural case of an uncountable inductive limit of
an increasing chain of separable stable and AF  ideals: we construct a GCR,  AF (in fact, scattered) subalgebra $\A$
of $\B(\ell_2)$,  which is the inductive limit of length $\omega_1$  of its separable stable ideals 
$\I_\alpha$ ($\alpha<\omega_1$) satisfying
$\I_{\alpha+1}/\I_\alpha\cong \K(\ell_2)$ for each $\alpha<\omega_1$, while $\A$ is not stable. 
The sequence $(\I_\alpha)_{\alpha\leq\omega_1}$  is the GCR composition
series of $\A$ which in this case coincides with the Cantor-Bendixson
composition series as a scattered $C^*$-algebra.
$\A$ has the property that all of its proper two-sided ideals are listed  as $\I_\alpha$s for some $\alpha<\omega_1$
and therefore the family of stable ideals of $\A$ has no  maximal element.

By taking $\A'=\A\otimes \K(\ell_2)$ we obtain a stable $C^*$-algebra with analogous composition series 
$(\J_\alpha)_{\alpha<\omega_1}$
whose ideals
$\J_\alpha$s are isomorphic to $\I_\alpha$s for each $\alpha<\omega_1$.
In particular, there are nonisomorphic scattered $C^*$-algebras whose GCR composition
series   $(\I_\alpha)_{\alpha\leq\omega_1}$ satisfy $\I_{\alpha+1}/\I_\alpha\cong \K(\ell_2)$
for all $\alpha<\omega_1$, for which the composition series differ first at $\alpha=\omega_1$.

\end{abstract}

\maketitle

\section{introduction}
\begin{definition}\label{elem-ess} Let $\A$ be a $C^*$-algebra and $\beta$ 
be an ordinal. A sequence  of ideals $(\I_\alpha)_{\alpha\leq\beta}$  of $\A$ is 
called an elementary essential composition series for $\A$
of length $\beta$ if and only if 

\begin{enumerate}[(a)]
\item $\I_0=\{0\}$, $\I_{\beta}=\A$, $\I_\alpha\subseteq \I_{\alpha'}$ for $\alpha\leq\alpha'\leq\beta$,
\item $\I_\lambda=\overline{\bigcup_{\alpha<\lambda}\I_\alpha}$ for all limit ordinals $\lambda\leq\beta$,
\end{enumerate}
For every $\alpha<\beta$
\begin{enumerate}[(c)]
\item[(c)]   $\I_{\alpha+1}/\I_\alpha$ is an essential ideal of  $\A/\I_\alpha$,
\item[(d)] $\I_{\alpha+1}/\I_\alpha$ is isomorphic to $\K(\ell_2)$.
%\item $\A$ is not $*$-isomorphic to $\A\otimes \K(\ell_2)$.

\end{enumerate}
\end{definition}
\begin{definition} A $C^*$-algebra $\A$ is called stable if it is isomorphic to
$\A\otimes \K(\ell_2)$.
\end{definition}

The purpose of this article is to 
prove the following:

\begin{theorem}\label{main0}
There is a nonstable $C^*$-subalgebra of $\B(\ell_2)$
which has an elementary essential composition series $(\I_\alpha)_{\alpha\leq\omega_1}$,
where $\omega_1$ is the first uncountable cardinal.
\end{theorem}
\begin{proof} Combine Theorems \ref{luzin-existence}, \ref{pairing} and \ref{main}. 
\end{proof}

Let us discuss several aspects of this construction. 
First recall that a
  $C^*$-algebra $\A$
is called approximately finite dimensional (AF) if $\A$ contains a directed family of finite-dimensional
subalgebras whose union is dense in $\A$, and it is called
locally-finite dimensional (LF) if every finite subset of $\A$ can be approximated from a finite-dimensional
subalgebra of $\A$. See  the paper \cite{farah-katsura} of Farah and Katsura for the
results showing the relations between these notions in the nonseparable case. 
In particular,   AF implies LF, and being  AF is equivalent to being LF for separable $C^*$-algebras by a result of Bratteli \cite{bratteli}.
More generally, for $C^*$-algebras of  density $\omega_1$, like our algebra from
Theorem \ref{main0}, by a result of Farah and Katsura (\cite{farah-katsura}) 
the two notions of AF and LF are equivalent as well. However, they also showed that the two notions are not equivalent in general.
 
Now let us list a few classical results concerning separable $C^*$-algebras which are relevant to our construction:
\begin{enumerate}
\item An extension of a separable AF-algebra by a separable AF-algebra algebra is a separable AF-algebra
 (Brown \cite{brown}, cf. \cite{davidson}).
\item Countable inductive limit of separable AF-algebras are  AF-algebras (Bratteli \cite{bratteli}, cf. \cite{davidson}).
\item An extension of a separable  stable AF-algebra by a separable stable AF-algebra is stable  
(Blackadar \cite{blackadar-af}, cf. 6.12 of \cite{rordam-stable}, 7.3 of \cite{cb}).
\item Countable inductive limits of separable stable algebras are  stable  (Hjelmborg and R\o rdam  4.1 of \cite{HR}).
\end{enumerate}

Combining these results one concludes that all ideals $\I_\alpha$  for $\alpha<\omega_1$ 
from the elementary essential composition series of our algebra $\A$ from Theorem \ref{main0}
are separable stable AF-algebras. So we observe a strong failure of the permanence of stability
in the context of the simplest uncountable inductive limits of separable AF-algebras:

\begin{theorem}\label{inductive-limit} There is an AF-algebra which is an uncountable inductive limit of an increasing chain of 
 separable stable 
AF-ideals
which is not stable.
\end{theorem}

One should mention here other recent
results concerning the failure 
of the permanence of stability and/or being AF for
quite fundamental nonseparable $C^*$-algebras: 
\begin{itemize}
\item There are $2$-subhomogenous non-AF extensions  
of nonseparable AF-algebras by AF-algebras (Theorem 1.12 of \cite{aiaus}).
\item There are  nonstable  $C^*$-algebras $\A$ satisfying 
the following short exact sequence:
$$0\rightarrow \mathcal K(\ell_2)\xrightarrow{\iota} \mathcal A 
\rightarrow\mathcal K(\ell_2({2^\omega}))\rightarrow 0,$$
where $\iota[\mathcal K(\ell_2)]$ is an essential ideal of
$\mathcal A$ (\cite{mrowka}).  
\end{itemize}
A partial version of Theorem \ref{inductive-limit} 
for an inductive limit  of length possibly bigger than
the first uncountable ordinal $\omega_1$ of possibly nonseparable $C^*$-algebras was obtained in Theorem 7.7. of \cite{cb}.

Another aspect of our construction is related to the class of scattered $C^*$-algebras.
They probably first appeared in a paper \cite{tomiyama} of Tomiyama, but were first explicitly defined
by H. Jensen in \cite{jensen}. Some of many equivalent conditions defining scattered $C^*$-algebras
are: every nonzero quotient has a minimal projection or the spectrum of every self-adjoint element is at most
countable (\cite{wojtaszczyk}) or every subalgebra is LF \cite{kusuda-af}. For a recent survey on
scattered $C^*$-algebras see \cite{cb}. Commutative scattered $C^*$-algebras
are exactly of the form $C_0(X)$ where $X$ is locally compact and scattered, i.e., its every
nonempty (closed) subset has a relatively isolated point. In fact, such spaces must be totally disconnected
and so, by the Stone duality they correspond exactly to superatomic Boolean algebras (\cite{roitman-handbook}).
In \cite{cb} we introduced a canonical
composition series $(\I_\alpha(\A))_{\alpha\leq ht(\A)}$ for a
 scattered $C^*$-algebra $\A$ which corresponds to the Cantor-Bendixson derivatives 
in the commutative case.  $\I_{\alpha+1}(\A)$ is defined by requiring
that  $\I_{\alpha+1}(\A)/\I_\alpha(\A)$ is the subalgebra generated in $\A/\I_\alpha(\A)$ 
by all minimal projections. In a scattered $C^*$-algebra the ideals $\I_{\alpha+1}(\A)/\I_\alpha(\A)$
 for $\alpha<ht(\A)$  are essential in $\A/\I_\alpha(\A)$ and
isomorphic to nondegenerate subalgebras of $\K(\ell_2(\kappa_\alpha))$ for some cardinals $\kappa_\alpha$.
Moreover the algebra $\A$  is scattered if and only if $\A=\I_{ht(\A)}(\A)$ for some ordinal $ht(\A)$
called the height of the scattered algebra. The width of  $\A$ is defined as 
$\sup\{\kappa_\alpha: \alpha<ht(\A)\}$. Following the commutative convention, a scattered
$C^*$-algebra of height $\omega_1$ and width $\omega$ is called thin-tall. Then
a superatomic Boolean algebra is thin-tall if its Stone space is thin-tall and a locally
compact  Hausdorff space $X$ is  thin-tall if $C_0(X)$ is a thin-tall $C^*$-algebra.

 Thin-tall locally compact Hausdorff spaces found an impressive amount
of applications in topology  usually as versions of the Ostaszewski space or the Kunen line
(\cite{ostaszewski}, \cite{roitman-toronto}, \cite{negrepontis}).
In fact the Banach spaces $C_0(X)$ where $X$ are thin-tall play an important role in 
the Banach space theory as well (\cite{kunen-shelah}, \cite{negrepontis}, \cite{biorthogonal}). 
Our construction shows that,  as in the commutative case,  maximally noncommutative 
thin-tall algebras do not need not be isomorphic to each other. To state precisely  a corollary of our
construction we need a notion of a fully noncommutative scattered  $C^*$-algebra:

\begin{definition}[\cite{cb} Definition 6.1] In the above notation a scattered $C^*$-algebra is fully noncommutative
if and only if the consecutive quotients $\I_{\alpha+1}(\A)/\I_\alpha(\A)$ 
for $\alpha<ht(\A)$ of the Cantor-Bendixson composition series $(\I_\alpha(\A))_{\alpha\leq ht(\A)}$
are isomorphic to the algebras of all compact operators $\K(\ell_2(\kappa_\alpha))$, 
respectively.
\end{definition}

\begin{lemma}\label{series-scattered}
Suppose that  $\beta$ is an ordinal,  $\A$ is a $C^*$-algebra and  $(\I_\alpha)_{\alpha\leq\beta}$
is a sequence of its ideals. $(\I_\alpha)_{\alpha\leq\beta}$ is
 an elementary essential 
composition series for $\A$ if and only if  $\A$ is scattered fully noncommutative
with Cantor-Bendixson composition series equal to $(\I_\alpha)_{\alpha\leq\beta}$
and $\I_{\alpha+1}/\I_\alpha\cong \K(\ell_2)$ for every $\alpha<\beta$.
\end{lemma}
\begin{proof}
For the forward implication, by Definition 1.3 and Theorem 1.4 of \cite{cb}
we need to note that  the ideal $\I_{\alpha+1}/\I_\alpha$ of  $\A/\I_\alpha$ 
is equal to the ideal $\I^{At}(\A/\I_\alpha)$ generated by the minimal projections of $\A/\I_\alpha$.
For this, by Theorem 1.2 of \cite{cb} it is enough to note that $\I_{\alpha+1}/\I_\alpha$ 
is essential in  $\A/\I_\alpha$  and isomorphic to a subalgebra of $\K(\ell_2)$ which is
the case by Definition \ref{elem-ess} (c), (d).

For the backward implication we apply Theorem 1.4 of \cite{cb} to obtain (a), (b) of 
Definition \ref{elem-ess}, Proposition 4.3 and Theorem 1.4 to obtain (c) and  the definition
of being fully noncommutative to obtain (d).
\end{proof}

It turns out that being fully noncommutative is equivalent to a strong noncommutativity
condition, namely that the center of the multiplier algebra of any quotient is trivial (Proposition 6.3 of \cite{cb}).
Also for a fully noncommutative scattered $C^*$-algebra  the Cantor-Bendixson composition series
coincides with the GCR composition series (Proposition 6.4 of \cite{cb}) and
full noncommutativity is equivalent to the stability for separable scattered $C^*$-algebras
such that $\kappa_\alpha=\omega$ for all $\alpha<ht(\A)$
(Lemma 7.3 of \cite{cb}).  So our construction yields the
following:

\begin{theorem} There are two scattered thin-tall fully noncommutative $C^*$-algebras $\A$ and $\B$
with the Cantor-Bendixson composition series $(\I_\alpha(\A))_{\alpha\leq\omega_1}$ and
$(\I_\alpha(\B))_{\alpha\leq\omega_1}$ such that $\I_\alpha(\A)$ is isomorphic to $\I_\alpha(\B)$
for every $\alpha<\omega_1$ but $\A$ is not isomorphic to $\B$, namely $\B$ is stable and $\A$
is not.
\end{theorem}
\begin{proof}
Let $\A$ be the algebra satisfying Theorem \ref{main0} with an elementary essential composition
series  $(\I_\alpha)_{\alpha\leq\omega_1}$. 
By Lemma \ref{series-scattered} the algebra $\A$ is scattered thin-tall fully noncommutative
whose Cantor-Bendixson composition series is $(\I_\alpha)_{\alpha\leq\omega_1}$.

Consider $\B=\A\otimes\K(\ell_2)$. It follows from Proposition 5.3 of \cite{cb} that $\B$ is scattered
of the same height $\omega_1$ whose Cantor-Bendixson composition series $(\I_\alpha(\B))_{\alpha\leq\omega_1}$ satisfies
 $\I_{\alpha}(\B)=\I_\alpha\otimes \K(\ell_2)$ for all $\alpha\leq\omega_1$. 
 But $\I_\alpha\otimes \K(\ell_2)$ is isomorphic to $\I_\alpha(\A)$ for $\alpha<\omega_1$
by the stability of $\I_\alpha(\A)$ as observed after (1) - (4).  Proposition 5.3. of \cite{cb} also implies
that $\I_{\alpha+1}(\B)/\I_\alpha(\B)$ is isomorphic to $\K(\ell_2)\otimes \K(\ell_2)\cong \K(\ell_2)$
for every $\alpha<\omega_1$  which gives
that $\B$ is scattered thin-tall fully noncommutative.  However, Theorem \ref{main0} yields that $\A$ is not stable
while $\B$ is trivially stable.
\end{proof}

In fact, our construction uses similar combinatorial ideas as the first absolute construction of
two nonisomorphic  thin-tall superatomic Boolean algebras from \cite{simon-weese} due to Simon and Weese
(cf. \cite{roitman-handbook}).  The latter corresponds to a locally compact thin-tall $X$ which cannot be
written as a disjoint union  $X=X_1\cup X_2$ where both of $X_1$ and $X_2$ are clopen and nonmetrizable.
On the other hand, $X\times\N$ is also thin-tall locally compact but can be written as a disjoint union as above.
In analogy to the above
property, in our algebra  $\A$ there are no two nonseparable subalgebras
$\A', \A''\subseteq \A$ such that $A'A''=0$ for all $A'\in \A'$ and $A''\in \A''$ (cf. 
Lemma \ref{orthogonal-sequence}, Theorem \ref{main}).
The combinatorial idea behind these examples is to use  a Luzin almost disjoint family of subsets
of $\N$ (\cite{luzin}, cf. \cite{hrusak}) to prevent
the algebra from splitting. Recall that a family $\mathcal F\subseteq \wp(\N)$ is called
Luzin if it has cardinality $\omega_1$, it is almost disjoint i.e., $A\cap B$ is finite for
any two distinct $A, B\in \mathcal F$ and there are no separations of uncountable subfamilies
i.e., given two disjoint uncountable $\mathcal F', \mathcal F''\subseteq \mathcal F$ 
there is no $C\subseteq \N$ such that $A\setminus C$ and $B\cap C$ are both finite
for all $A\in \mathcal F'$ and $B\in \mathcal F''$.
In our case we need additional properties of a Luzin family (Theorem \ref{complicated-ad}) which
are combinatorially interesting by themselves 
%(\cite{luzin-separated})
 and are published elsewhere (\cite{hgk}).
Such a family of subsets of $\N$ yields a system of noncompact operators  in $\B(\ell_2)$
which we call a Luzin blockwise system of almost matrix units (Definitions \ref{blockwise-def} and \ref{blockwise-luzin-def}).
Its behaviour can be expressed in a  manner similar to the commutative one described above
($=^\K$ denotes the equality modulo compact operators):

\begin{theorem}
There is a sequence $(\A_\alpha)_{\alpha<\omega_1}$ of $C^*$-subalgebras of $\B(\ell_2)$
which are all isomorphic to $\K(\ell_2)$ and
which are pairwise almost orthogonal, i.e., $AA'=^\K0$ for all $A\in \A_\alpha$, $A'\in \A_{\alpha'}$
for any $\alpha<\alpha'<\omega_1$ with the following property:

Given any two
uncountable $X, Y\subseteq\omega_1$ and any choice of
$A_\alpha\in \A_\alpha$ for $\alpha\in X$ and $B_\beta\in \A_\beta$ for $\beta\in Y$ there is no
projection $P\in\B(\ell_2)$ satisfying
$$PA_\alpha =^\K A_\alpha \ \hbox{for all}\ \alpha\in X\ \  \hbox{and}\ \ 
 PB_\beta =^\K 0\  \hbox{for all}\ \beta\in Y.
\leqno (\perp)$$
\end{theorem}
\begin{proof}
We claim that  the algebras $\A_\alpha=C^*(\mathcal S_\alpha)$ for a Luzin blockwise system of almost matrix units
defined in Definitions \ref{blockwise-def} and \ref{blockwise-luzin-def}, which exists by
Proposition \ref{luzin-existence}, satisfy the theorem. Let $X, Y$ be as in the theorem 
and suppose that there is  a  projection $P\in\B(\ell_2)$ as in ($\perp$). Let
$V_\alpha, U_\beta\in \K(\ell_2)$  for $\alpha\in X$ and $\beta\in Y$ be such that
$A_\alpha^* P-A_\alpha^*=V_\alpha$ for $\alpha\in X$ and 
$(1-P)B_\beta -B_\beta =U_\beta$ for $\beta\in Y$. So $(A_\alpha^*+V_\alpha)(B_\beta+U_\beta)=0$
for $\alpha\in X, \beta\in Y$.
Using the separability of $\C$ and of $\K(\ell_2)$
by thining out $X$ and $Y$ to uncountable subsets we may assume that there are $W_1, W_2\in \K(\ell_2)$
such that $\|(A_\alpha^*-W_1)(B_\beta-W_2)\|< \|A_\alpha^*\|\|B_\beta\|/2$ for all
$\alpha\in X$ and $\beta\in Y$. But this contradicts the Luzin property from Definition\ref{blockwise-luzin-def}.
\end{proof}

We should note that almost disjoint families
and in particular  Luzin families  recently found several applications in constructions
of interesting noncommutative objects (\cite{wofsey, farah-wofsey, ad, vaccaro}). Ours seems to be the first application where
one considers collections of subalgebras rather than collections of elements of a $C^*$-algebra.

Today the diversity of thin-tall algebras in the commutative case is much better understood than at the moment of
publication of \cite{simon-weese}. In \cite{roitman-auto} Roitman showed that 
  it is consistent that there are
$2^{\omega_1}$ (as many as possible) pairwise non-isomorphic thin-tall  superatomic Boolean algebras.
 In \cite{dow-simon}
Dow and Simon distinguished  in ZFC $2^{\omega_1}$
 nonisomorphic thin-tall superatomic Boolean algebras by analyzing the groups
 of automorphisms of such algebras. It would be interesting
 also to study these groups in the fully noncommutative case.

Another aspect of our construction is related to 
the structure of the family of all two-sided ideals of $\A$. Lemma 6.2 of \cite{cb} implies that
all two-sided ideals of $\A$ are among the ideals $\I_\alpha$ for $\alpha\leq \omega_1$. In particular,
they form a continuous chain where all elements are stable for $\alpha<\omega_1$ and 
$\I_{\omega_1}=\A$ is not stable, so we obtain:

\begin{theorem} There exists a $C^*$-algebra, where the family of all stable ideals
has   no maximal 
element. In particular, this
family does not have the greatest element.
\end{theorem}

This gives a negative answer  to  Question 6.5 of \cite{rordam-stable} (only) in the nonseparable case,
which asks if every $C^*$-algebra has the greatest stable ideal.
In fact, in every separable $C^*$-algebra  there are  maximal elements in the family
of stable ideals
(see the discussion after 6.5 of \cite{rordam-stable}). Clearly, the additional feature of this
answer is the simplicity of the composition series of the algebra.

The purpose of Section 2 is to introduce the appropriate terminology and to prove the existence
of a Luzin blockwise system of almost matrix units (Proposition \ref{luzin-existence}).
 In  Section 3 we show how to connect a Luzin blockwise system of almost matrix units
 with an elementary essential composition series (Theorem \ref{pairing}).
 The main relation between these two objects is called domination (Definition
\ref{def-domination}). Section 4 is devoted to the proof of the nonstablity
of the algebra constructed in Section 3.

The terminology should be standard, e.g., like in  \cite{Blackadar, davidson, murphy}. We list here
some possible exceptions. $\tilde \A$ stands for the unitization of a $C^*$-algebra $\A$.
$\K(\ell_2)$ denotes the algebra of compact operators on  the separable Hilbert space $\ell_2$
and $\B(\ell_2)$ the algebra of all bounded operators on $\ell_2$.
For $A, B\in \wp(\N)$  we use the notation $A\subseteq^{Fin}B$ to mean that $B\setminus A$ is finite,
similarly $A=^{Fin} B$ if $A\subseteq^{Fin} B$ and $B\subseteq^{Fin} A$.
For $A, B\in \B(\ell_2)$ we use the notation $A=^\K B$ to mean
$A-B\in \K(\ell_2)$. $\cong$ stands for  the isomorphism relation  of $C^*$-algebras by which we always mean
the $*$-isomorphism relation.
Two projections $P, Q\in \B(\ell_2)$ are {almost orthogonal}
if and only if $PQ=^\K0$ (cf. \cite{wofsey}). $\delta_{x, y}$ stands for the Kronecker delta.
A {system of matrix units} in a $C^*$-algebra  is a family of  its nonzero elements
$\mathcal T=\{T_{j, i}: i, j\in \N\}$, 
such that 
for every $m, n, i, j \in \N$ we have
\begin{itemize}
 \item $T_{j, i}^* =   T_{i, j}$ and 
 \item $T_{n, m}\ T_{j, i} =\delta_{m, j}T_{n, i}$.
 \end{itemize}

For a set of operators $\mathcal S \subseteq \B(\ell_2)$ let $ C^*(\mathcal S)$ denote the $C^*$-subalgebra of 
$\mathcal B(\ell_2)$ generated by
the operators in $\mathcal S$.

\section{Blockwise systems of almost  matrix units}

Let $(\lambda_\alpha)_{\alpha<\omega_1}$ be 
the strictly increasing sequence of all countable limit ordinals (including $0$).
We introduce the following  notations:
\begin{itemize}
\item $\Lambda_\alpha=[\lambda_\alpha, \lambda_{\alpha}+\omega)\times [\lambda_\alpha, \lambda_\alpha+\omega)$
 for each $\alpha<\omega_1$,
\item $\Lambda=\bigcup_{\alpha<\omega_1}\Lambda_\alpha$.
\end{itemize}

\begin{definition}\label{blockwise-def} Suppose that $\mathcal S=(S_{\eta, \xi}: (\xi, \eta)\in \Lambda)$
is a system of noncompact operators in $\B(\ell_2)$.  We say that $\mathcal S$ is a {blockwise system of almost
matrix units} whenever the following are satisfied:
\begin{enumerate}
\item $\mathcal S_{\alpha}=\{S_{\eta, \xi}: (\xi, \eta)\in \Lambda_\alpha\}$ is 
a system of matrix units in $\B(\ell_2)$ for every $\alpha<\omega_1$,
\item $\{S_{\xi, \xi}: \xi\in \omega_1\}$ is a family of pairwise almost orthogonal projections.
\end{enumerate}
We say that $\mathcal S$ is {separated} by 
a  sequence of projections $( P_\alpha: \alpha<\omega_1)\subseteq \B(\ell_2)$
whenever the following hold:
\begin{enumerate}
\item[(3)] $P_{\alpha'} P_\alpha=^\K P_{\alpha'}$ 
for all $\alpha'\leq \alpha<\omega_1$,
\item[(4)] $P_{\alpha}S_{\xi, \eta} P_{\alpha}
=^\K S_{\xi, \eta}$ for each $(\xi, \eta)\in \Lambda_{\alpha'}$ 
for each $\alpha'<\alpha<\omega_1$,
\item[(5)] $P_\alpha^\perp S_{\xi, \eta} P_\alpha^\perp
= S_{\xi, \eta}$ for each $(\xi, \eta)\in \Lambda_\alpha$ 
for each $\alpha<\omega_1$.
\end{enumerate}
\end{definition}

\begin{definition}\label{blockwise-luzin-def} 
Suppose that $\mathcal S=(S_{\eta, \xi}: (\xi, \eta)\in \Lambda)$ is a blockwise system
of almost matrix units. We say that  $\mathcal S$ is  {Luzin} if given
\begin{enumerate}
\item  two uncountable  subsets $X, Y$ of $\omega_1$,
\item $A_\alpha\in  C^*(\mathcal S_{\alpha})$
for each $\alpha\in X$,
\item $B_\alpha\in  C^*(\mathcal S_{\alpha})$
for each $\alpha\in Y$,
\item  $\varepsilon>0$,
\item $W_1, W_2\in \K(\ell_2)$,

\end{enumerate}
there are $\alpha \in X$ and $\beta\in Y$ such that
$$\|(A_{\alpha}-W_{1})(B_{\beta}-W_{2})\|\geq\|A_{\alpha}\| \|B_{\beta}\|-\varepsilon.$$
\end{definition}

We will use the almost disjoint family as in Theorem \ref{complicated-ad} to show that 
Luzin blockwise systems of almost matrix units exist and they can be separated by families of projections. 
First we need the following lemma.

\begin{lemma}\label{kronecker} Suppose that $n\in \N$ and that $\mathcal T=\{T_{m,k}: k, m\leq n\}\subseteq \B(\ell_2)$
and $\mathcal S=\{S_{j, i}: i, j\leq n\}\subseteq \B(\ell_2)$ are two finite systems of matrix units.
Suppose that there are pairwise orthogonal norm one vectors $(e^k_i: i,k\leq n)$ such that

\begin{enumerate}
%\item $(e^k_i: i\leq n)$ is included in the range of $T_{k, k}$ for each $k\leq n$,
%\item $(e^k_i: k\leq n)$ is included in the range of $S_{i, i}$ for each $i\leq n$,
\item  $T_{m, k}(e^{k'}_i)=\delta_{k, k'}e^m_i$ for all $i, k, k', m\leq n$,
\item  $S_{j, i}(e^k_{i'})=\delta_{i, i'}e^k_j$ for all $i, i', j, k\leq n$.
\end{enumerate}
Let $A\in C^*(\mathcal T)$ and
$B\in C^*(\mathcal S)$.
Then $\|AB\|=\|A\|\|B\|$. Moreover this fact is witnessed by a norm one
vector from $span(e^k_i: i, k \leq n)$.
\end{lemma}
\begin{proof} 
This follows from elementary properties of tensor products, but
we present a shorter complete proof producing the required vector.
As $\mathcal T=\{T_{m,k}: k, m\leq n\}$ is a system of matrix units, the algebra $C^*(\mathcal T)$ 
is isomorphic to the algebra $n\times n$ matrices and therefore is simple.
Hence,  for each $i\leq n$ the restriction of elements of $C^*(\mathcal T)$ to their invariant
subspace $\mathcal H_i=span\{e^k_i: k\leq n\}$, as a nonzero homomorphism,  is an isomorphism  of $C^*(\mathcal T)$
into $\B(\mathcal H_i)$.
%, but considering the dimensions of the algebras, we conclude that
%it is onto $\B(\mathcal H_i)$
It follows that there is  $x=(x_1,\dots, x_n)\in\ell_2^n$
such that $v_i=\sum_{k\leq n} x_ke^k_i$ is of norm one and
$A(v_i)=v_i'=\sum_{k\leq n} x_k'e^k_i$ and
$\|v_i'\|=\|A\|$ for each $i\leq n$. Likewise, there is $y=(y_1, \dots, y_n)\in\ell_2^n$
such that $w_k=\sum_{j\leq n} y_je^k_j$ is of norm one 
and
$B(w_k)=w_k'=\sum_{k\leq n} y_j'e^k_j$ and  $\|w_k'\|=\|B\|$ for each $k\leq n$.

By a direct calculation we note that vectors of $\ell_2$ of the form
$\sum_{j, k\leq n}\alpha_k\beta_je^k_j$ have norms equal to 
the product $\|(\alpha_1,\dots\alpha_n)\|_{\ell_2^{n}}\|(\beta_1,\dots\beta_n)\|_{\ell_2^{n}}$.
Consider norm one element
$z=
\sum_{j, k\leq n}x_ky_je^k_j$.
We have
$$BA(z)=B(\sum_{j\leq n} y_j A(\sum_{k\leq n}x_ke^k_j))=
B(\sum_{j\leq n} y_j (\sum_{k\leq n}x_k'e^k_j))=$$
$$=\sum_{k\leq n}x_k'B(\sum_{j\leq n} y_je^k_j)
=\sum_{k\leq n}x_k'(\sum_{j\leq n} y_j'e^k_j)=\sum_{j, k\leq n}x_k'y_j'e^k_j.$$
So $\|BA(z)\|=\|A(v_i)\|\|B(w_k)\|=\|A\|\|B\|$ for any $i, k\leq n$. Similarly $\|AB(z)\|=\|A\|\|B\|$, 
which completes the proof.

\end{proof}

\begin{theorem}[\cite{hgk}]\label{complicated-ad} There are families $(X_\alpha: \alpha<\omega_1)$, 
$(Y_\alpha: \alpha<\omega_1)$ of infinite subsets of $\N$ and
 bijections $x^\alpha:\N\times\N\rightarrow X_\alpha$
 for each $\alpha<\omega_1$
such that 
\begin{enumerate}
\item $X_\beta\cap X_\alpha=^{Fin}\emptyset$ for all $\beta<\alpha<\omega_1$,
\item $Y_\beta\subseteq^{Fin} Y_\alpha$ for all $\beta<\alpha<\omega_1$,
\item $X_\beta\subseteq^{Fin} Y_\alpha$ for all $\beta<\alpha<\omega_1$,
\item $X_\alpha\cap Y_\alpha=\emptyset$ for all $\alpha<\omega_1$,
\item For every $\alpha<\omega_1$ and every $k\in \N$ for all
but finitely many $\beta<\alpha$ there are $m_1< ...<m_k$ and $n_1< ...< n_k$ 
and $l_{i, j}\in \N$ such that 
$$x^\alpha({i,n_j})=l_{i, j}=x^\beta({j, m_i})$$
for all $1\leq i, j\leq k$.
\end{enumerate}
\end{theorem}

\begin{proposition}\label{luzin-existence} There is a Luzin blockwise system of almost matrix units which
is separated by a family of projections.
\end{proposition}
\begin{proof}
Let  
$(X_\alpha: \alpha<\omega_1)$ be the almost disjoint family with
enumerations $(x^\alpha_{n, m}: \alpha<\omega_1, n, m\in \N)$ 
which is separated by a family  $(Y_\alpha: \alpha<\omega_1)$ as in Theorem \ref{complicated-ad},
where $x^\alpha_{n, m}=x^\alpha(n, m)$ for all $\alpha<\omega_1, n, m\in \N$. 
Fix an orthogonal  basis $(e_n: n\in \N)$ of $\ell_2$ and define the following
operators diagonal with respect to this basis:
\begin{itemize}
\item $P_\alpha(e_n)=\chi_{Y_\alpha}(n)e_n$,
\item $S_{\lambda_\alpha+k, \lambda_\alpha+k}(e_n)=\chi_{\{x^\alpha_{k, i}: i\in \N\}}(n)e_n$,
\end{itemize}
for all $\alpha<\omega_1$, where $\chi_X$ denotes the characteristic function
of $X\subseteq \N$. Moreover for every 
$\alpha<\omega_1$ and every $m,k \in \N$ define the partial isometry $S_{\lambda_\alpha+k, \lambda_\alpha+m}$ by
\begin{itemize}
\item $S_{\lambda_\alpha+k, \lambda_\alpha+m}(e_{x^\alpha_{m, i}})=e_{x^\alpha_{k, i}}$ for every $i\in \N$,
\item $S_{\lambda_\alpha+k, \lambda_\alpha+m}(e_{n})=0$ if $n$ is not of the form $x^\alpha_{m, i}$
for some $i\in \N$.
\end{itemize}
It is immediate from Theorem \ref{complicated-ad} (1) - (4) that
$\mathcal S=\{S_{\eta, \xi}: (\xi, \eta)\in \Lambda\}$ is a blockwise system of almost matrix units
which is separated by $(P_\alpha: \alpha<\omega_1)$. We will use Theorem \ref{complicated-ad} (5)
to conclude that it is  Luzin. 

So fix
  two uncountable  subsets $X, Y$ of $\omega_1$ and operators
 $A_\alpha\in  C^*(\mathcal S_{\alpha})$
for each $\alpha\in X$, 
 $B_\alpha\in  C^*(\mathcal S_{\alpha})$
for each $\alpha\in Y$,
 $\varepsilon>0$ and two compact operators
 $W_1, W_2\in \K(\ell_2)$. We may assume that $\varepsilon<1$.

By approximating $W_1$ and $W_2$ we may assume that
there is $k_1\in \N$ such that $\langle W_1(e_n), e_{n'}\rangle=0=\langle W_2(e_n), e_{n'}\rangle$
whenever $n, n' \geq k_1$.
By passing to smaller uncountable subsets of $X$ and $Y$ respectively
we may assume that there is $M>1$  such that $\|A_\alpha\|<M$ for
all $\alpha\in X$ and $\|B_\alpha\|<M$ for all $\alpha\in Y$. Passing further
to uncountable subsets of $X$ and $Y$ we may assume that there is $k_2\in \N$ such that for every $\alpha \in X$ and $\beta\in Y$ 
there are 
 $k_2\times k_2$ matrices $(a_{m,n})_{m,n\leq k_2}$,
 and  $(b_{m,n})_{m,n\leq k_2}$ such that
$$\|A_\alpha'-A_\alpha\|<\delta, \qquad \|B_\beta'-B_\beta\|<\delta,$$ 
for some fixed $\delta=\delta(\varepsilon, M)>0$, where
$$A_\alpha'=\sum_{n, m<k_2}a_{m,n}S_{\lambda_\alpha+m, \lambda_\alpha+n}$$
for  each  $\alpha\in X$ and 
$$B_\beta'=\sum_{n, m<k_2}b_{m,n}S_{\lambda_\beta+m, \lambda_\beta+n}$$
for  each $\beta\in Y$.
Let $\alpha\in X$ be such that $Y\cap\alpha=Y\cap\{\beta:\beta<\alpha\}$ is infinite. 
By Theorem \ref{complicated-ad} (5)
there is a finite $F\subseteq\alpha$ such that whenever $\beta\in (Y\cap\alpha)\setminus F$, then
there are $m_1< ...<m_{k_1+k_2}$ and $n_1< ...< n_{k_1+k_2}$ 
and $l_{i, j}\in \N$ such that 
$$x^\alpha_{i,n_j}=l_{i, j}=x^\beta_{j, m_i}$$
for all $1\leq i, j\leq k_1+k_2$. 
Let $G\subseteq\{1, ..., k_1+k_2\}$  be of size $k_2$ such that $x^\alpha_{i, n_j}\not\in \{1, ..., k_1\}$
for $i\in G$ and any $1\leq j\leq k_1+k_2$.
Now note that $A_\alpha'$ and $B_\beta'$ satisfy the hypothesis
of Lemma \ref{kronecker} on the finite dimensional spaces
spanned by $\{e_{x^\alpha_{i, n_j}}: i, j\in G\}$. By the choice of $G$
the operators $W_1, W_2$ are null on this subspace. So Lemma
\ref{kronecker} implies that
$\|(A_\alpha'-W_1)(B_\beta'-W_2)\|=\|A_\alpha'\|\|B_\beta'\|$. 
This means that if $\delta$ is sufficiently small, then
$$\|(A_{\alpha}-W_{1})(B_{\beta}-W_{2})\|\geq\|A_{\alpha}\| \|B_{\beta}\|-\varepsilon.$$

\end{proof}

\section{Dominating a blockwise system of almost matrix units by a representing sequence}

The following are two main definitions of this section:

\begin{definition}\label{def-representing}
 Let $\A$ be a $C^*$-algebra with
 an elementary essential composition series $(\I_\alpha)_{\alpha\leq\beta}$
 for some ordinal $\beta$, and let $\pi_\alpha: \A\rightarrow \A/\I_{\alpha}$ for $\alpha<\beta$
be the quotient homomorphisms.
 A system  $\mathcal T= (T_{\alpha+1, m, n}: n, m\in \N, \alpha<\beta)$ of elements
 of $\A$
is called a {representing sequence} for $\A$
if for each $\alpha<\beta$ the following hold:
\begin{enumerate}
\item $(T_{\alpha+1, m, n}: n, m\in \N)$ is a 
system of matrix units in $\A$.
\item $( \pi_\alpha(T_{\alpha+1, m, n}): n, m\in \N)$ is a system 
of matrix units in $\A/\I_{\alpha}$ which generates
 $\I_{\alpha+1}/\I_\alpha$.

\end{enumerate}
We also say that $(T_{\alpha+1, m, n}: m, n\in \N)$ represents the $(\alpha+1)$-th
level of $\A$.
\end{definition}

\begin{definition}\label{def-domination}
Suppose that $\mathcal S=\{S_{\eta, \xi}: (\xi, \eta)\in \Lambda\}$ is a blockwise system of almost 
matrix units which is separated by a family of 
projections $\mathcal P= (P_\alpha: \alpha<\omega_1)$. Let
$\mathcal A\subseteq \B(\ell_2)$ be a   $C^*$-algebra with an elementary essential
composition series with
a representing sequence $\mathcal T=(T_{\alpha+1, m, n}: n, m\in \N, \alpha<\omega_1)$.
We say that $\mathcal T$ {dominates} $\mathcal S$ if for every $0<\alpha<\omega_1$, $m,n\in\N$, we have
$$T_{\alpha+1, m, n}= P_\alpha T_{\alpha+1, m, n} P_\alpha+S_{\lambda_\alpha+m, \lambda_\alpha+n}. \leqno (*)$$

\end{definition}

The reason that the case $\alpha= 0$ is excluded in the definition above, is that in Theorem \ref{pairing} we would like 
$(T_{1, m, n}: n, m\in \N)$ to be the standard system of matrix units for $\K(\ell_2)$, and 
$S_{\lambda_0+n,\lambda_0+m}$ the way constructed in Proposition \ref{luzin-existence}
are noncompact operators, hence do not satisfy the condition of Definition \ref{def-domination} for $\alpha=0$.

We will show in Theorem \ref{pairing} that given any blockwise system of almost matrix units $\mathcal S$
we can find a $C^*$-algebra $\A$ with an essential elementary 
composition series with a representing sequence $\mathcal T$ which
  dominates $\mathcal S$. 
For this we need some lemmas. 

\begin{definition}Let $\A$ be a $C^*$-subalgebra of $\B(\ell_2)$ with
 an elementary essential composition series $(\I_\alpha)_{\alpha\leq\beta}$
 for some ordinal $\beta$, and let $\pi_\alpha: \A\rightarrow \A/\I_{\alpha}$ for $\alpha<\beta$
be the quotient homomorphisms. For each element $A\in \A$ we define its height $ht(A)$ by
$$ht(A)=\min\{\alpha\leq\beta: A\in \I_\alpha\}.$$
\end{definition} 

\begin{lemma}\label{lemma-domination} Suppose that $\mathcal S, \mathcal P, \A$ and $\mathcal T$ are  as 
in Definition \ref{def-domination} and $0\leq\alpha<\omega_1$. 
If $\mathcal T$ dominates $\mathcal S$, then for every $A \in \A$ of height $\leq \alpha+1$  we have
$A=^\K P_\alpha A P_\alpha +C$, where  $C\in  C^*(\mathcal S_{\alpha})$.
If $\alpha>0$ and $ht(A)=\alpha+1$, then $C\not=0$.
\end{lemma}

\begin{proof}
 The  elements $A\in \B(\ell_2)$ of the form $A=^\K P_\alpha AP_\alpha+C$ where
$\alpha<\omega_1$ and $C\in C^*(\mathcal S_\alpha)$ form a subalgebra of $\B(\ell_2)$
because $P_\alpha^\perp CP_\alpha^\perp=C$ for each $C\in C^*(\mathcal S_\alpha)$ by
Definition \ref{blockwise-def} (5).

 By Definition \ref{blockwise-def} items (3), (4) we have that 
  $P_{\alpha'} P_\alpha=^\K P_{\alpha'}$ and moreover
 $P_{\alpha}S_{\lambda_{\alpha'}+m, \lambda_{\alpha'}+n} P_{\alpha}
=^\K S_{\lambda_{\alpha'}+m, \lambda_{\alpha'}+n}$ 
for each $\alpha'<\alpha<\omega_1$ and each $n, m\in \N$.
So by ($*$)  we have  $P_\alpha T_{\alpha'+1, m, n} P_\alpha=^\K T_{\alpha'+1, m, n}$ for each $\alpha'<\alpha$
  and each $n, m\in \N$. By  ($*$)
$T_{\alpha+1, m, n}= P_\alpha T_{\alpha+1, m, n} P_\alpha
+C$ for $C\in C^*(\mathcal S_{\alpha})$. 
The first part of the lemma follows from the fact that $T_{\alpha'+1, m, n}$s for $\alpha'\leq\alpha$
  and each $n, m\in \N$ generate $\I_{\alpha+1}$.

  By Definition \ref{def-representing} (2) $ht(A)=\alpha+1$ implies that there is 
  $B\in  C^*(\mathcal T_\alpha)\setminus \{0\}$
  such that $A-B\in \I_\alpha$, where $\mathcal T_\alpha=\{T_{\alpha+1, m, n}: n, m\in \N\}$. 
  As noted in the first part of the proof $P_\alpha D P_\alpha=^\K D$ for each $D\in \I_\alpha$.
  So for the second part of the lemma it is enought to prove that $P_\alpha^\perp B P_\alpha^\perp\not\in \K(\ell_2)$.
  
  By ($*$)
  and Definition \ref{blockwise-def} (5) the
  range of $P_\alpha^\perp$ is invariant for $ C^*(\mathcal T_\alpha)$, and restricting elements of 
  $ C^*(\mathcal T_\alpha)$ to
  it is an isomorhpism. By ($*$) and the hypothesis that $\alpha>0$
  the range of this isomorphism is $ C^*(\mathcal S_\alpha)$ which
  satisfies $ C^*(\mathcal S_\alpha)\cap \K(\ell_2)=\{0\}$. It follows that $P_\alpha^\perp BP_\alpha^\perp\not
  \in \K(\ell_2)$ as required.
\end{proof}

\begin{lemma}\label{auto-unitary}
Suppose  $\A, \A'\subseteq \B(\ell_2)$ are 
  $C^*$-subalgebras of $\B(\ell_2)$ which both contain $\K(\ell_2)$. Then
  every isomorhism $\Phi:\A\rightarrow \A'$ is given by $\Phi(A)=U^*AU$
  for all $A\in \A$ for some fixed unitary $U\in \B(\ell_2)$.
  \end{lemma}
  \begin{proof}
  Note that
minimal projections of $\A$ and of $\A'$ are one dimensional, 
and $\Phi$ preserves minimal projections, which means that $\Phi[\K(\ell_2)] = \K(\ell_2)$.
Every automorphism of $\K(\ell_2)$ is induced by conjugating with a unitary in $\B(\ell_2)$  (Theorem 2.4.8 of \cite{murphy}).
 So for some unitary $U$ we have $\Phi(T) = U^* T U$, for every $T$ in $\K(\ell_2)$. Fix an orthogonal basis $\{e_n: n\in \N\}$ for $\ell_2$
and let $P_n$ for
$n\in\N$ denotes the projection on the one dimensional subspace spanned by $e_n$. Let $Q_n = U P_n U^*$ for every $n\in \N$. 
For every   $T\in \A$ and $m,n$ we have 
\begin{align*}
(\Phi(T) e_n , e_m) &= (P_m \Phi(T)P_n e_n, e_m) = (\Phi(Q_m T Q_n)e_n , e_m)\\
&= (U^* Q_m T Q_n U e_n , e_m) = (P_m U^* T U P_n e_n , e_m)\\
&=(U^*TU e_n, e_m).
\end{align*}
Therefore $\Phi (T) = U^* T U$, for every $T\in \A$.
\end{proof} 

%\begin{lemma}\label{ess-ess}
%Suppose that $\beta$ is an ordinal and $(I_\alpha)_{\alpha\leq \beta}$  is an essential
%composition series. Then $I_\alpha$ is an essential ideal of $I_\beta$ for every $\alpha\leq\beta$
%\end{lemma}
%\begin{proof}
%If $\alpha=\alpha'+1$, then by Definition \ref{elem-ess} (c) we have that
%$I_\alpha/I_{\alpha'}$ is essential in $I_\beta/I_{\alpha'}$ and so
%the esentiality of $I_\alpha$ follows. If $\alpha$ is a limit ordinal
%\end{proof}

\begin{lemma}\label{stable-extension0} Suppose that $\beta$ is an ordinal and 
$(\I_\alpha)_{\alpha\leq\beta}$ is an elementary essential composition series for 
 a stable   $C^*$-algebra $\A$.
Then there is a  stable   $C^*$-algebra
$\B$ with an essential elementary composition series
$(\J_\alpha)_{\alpha\leq\beta+1}$ satisfying $\J_\alpha=\I_\alpha$ for $\alpha\leq\beta$
and $\J_{\beta+1}=\B$ with a sequence $\{T_{\beta+1, m, n}: n, m\in \N\}$
representing the $\beta+1$-th level of $\B$.
\end{lemma}
\begin{proof}
By Lemma \ref{series-scattered} $\A$ is scattered $C^*$-algebra with the Cantor-Bendixson 
composition series $(\I_\alpha)_{\alpha\leq\beta}$. So by Lemma 7.5 of \cite{cb}
and again Lemma \ref{series-scattered} the algebra
$\B=\tilde \A\otimes\K(\ell_2)$, where $\tilde\A$ is the unitization of $\A$, is the required algebra 
with the identification of $\A$ and $\A\otimes \K(\ell_2)$ which is justified
by the stability of $\A$. 
$T_{\beta+1, m, n}=1\otimes e_{m,n}$, where $e_{m, n}$s are the standard
matrix units in $\K(\ell_2)$.
\end{proof}

\begin{lemma}\label{stable-extension} Suppose that $\beta<\omega_1$ and 
$(\I_\alpha)_{\alpha\leq\beta}$ is an elementary essential composition series for 
 a stable   $C^*$-subalgebra $\A$ of $\B(\ell_2)$ such that $\I_1=\K(\ell_2)$.
Then there is a  stable   $C^*$-subalgebra
$\B$ of  $\B(\ell_2)$ with an elementary essential composition series
$(\J_\alpha)_{\alpha\leq\beta+1}$ satisfying $\J_\alpha=\I_\alpha$ for $\alpha\leq\beta$
and $\J_{\beta+1}=\B$ with a sequence $\{T_{\beta+1, m, n}: n, m\in \N\}$
representing the $\beta+1$-th level of $\B$
\end{lemma}
\begin{proof}
We need to find a $C^*$-subalgebra $\B$ of $\B(\ell_2)$ such that
$\A\subseteq \B$ and $\B$ satisfies Lemma \ref{stable-extension0}.
First consider a $\B'$ which satisfies Lemma \ref{stable-extension0}, not necessarily a subalgebra of
$\B(\ell_2)$.
The second step is to obtain  an algebra  $\B''$ isomorphic to $\B'$ and satisfying $\K(\ell_2)\subseteq \B''\subseteq \B(\ell_2)$.
To get it note that
 $\K(\ell_2)$ must be an essential ideal 
 of $\B'$ by Definition \ref{elem-ess}, so,  since $\B(\ell_2)$ is  the
multiplier algebra of $\K(\ell_2)$,
 there is an embedding $\Phi:  \B'\rightarrow \B(\ell_2)$  with image $\phi[\B']=\B''$ such that
$\Phi[\K(\ell_2)]=\K(\ell_2)$.  Let $\A''=\Phi[\A]$.
By Lemma \ref{auto-unitary} there is a  unitary $U\in \B(\ell_2)$ such that
the conjugation by $U$ is an isomorphism from $\A''$ onto $\A$. Since the conjugation by $U$
is an automorphism of the entire $\B(\ell_2)$ we conclude that
 $\B=\{U^*BU: B\in \B''\}$ works. 
\end{proof}

\begin{lemma}\label{corner-composition} Let $\beta$ be an ordinal. Suppose 
that a $C^*$-algebra $\A\subseteq \B(\ell_2)$ has
an elementary essential composition series $(\I_\alpha)_{\alpha\leq\beta}$,
with $\I_1=\K(\ell_2)$ and that there is an infinite rank projection
 $P\in \B(\ell_2)$  such that $A-PAP\in \K(\ell_2)$
for each $A\in \A$.  Then $P\A P$ is a $C^*$-subalgebra of $\B(\ell_2)$
 with
an elementary essential composition series $(P\I_\alpha P)_{\alpha\leq\beta}$,
 and
 $\A$ is generated by $\K(\ell_2)$ and $P\A P$. 
% If $\A$ is fully non-commutative with representing sequence  $\mathcal T=(T_{\beta+1, m, n}: n, m\in \N, \beta< \alpha)$, then $P\mathcal T P$ represents $P\A P$.
\end{lemma}
\begin{proof}
%Let $\pi_\beta: \A \rightarrow \A / \I_\beta(\A)$ and $\pi_\beta': P\A P \rightarrow P\A P/ P\I_\beta(\A)P$ be the canonical quotient maps.
Since  $A-PAP\in \K(\ell_2)$
for each $A\in \A$ and $\K(\ell_2) \subseteq \A$, we have that $P\A P \subseteq \A$.
$P\K(\ell_2)P$ is an essential ideal of $P\B(\ell_2)P$ and so an essential ideal of $P\A P$
which is isomorphic to $\K(\ell_2)$ since $P$ has infinite rank.

 Let $\pi_1: \A\rightarrow \A/\K(\ell_2)$ and $\sigma_1: P\A P\rightarrow P\A P/P\K(\ell_2)P$
 be the quotient maps. $\Psi:\A/\K(\ell_2)\rightarrow P\A P/P\K(\ell_2)P$  given by
 $\Psi(\pi_1(A))=\sigma_1(PA P)$ for $A\in\A$ is a well defined isomorphism  since e.g.,  $PABP-PAPBP\in \K(\ell_2)$
 for any $A, B\in\A$. It is clear that $\A/\I_1$ has an elementary
 essential composition series $(\I_\alpha/\I_1)_{\alpha\leq\beta}$ and so
 $P\A P/P\K(\ell_2)P$ has $(\Psi[\I_\alpha/\I_1])_{\alpha\leq\beta}$ as such a series.
 It follows that $P\A P$ has an essential composition series $(\J_\alpha)_{\alpha\leq\beta}$
 where $\J_1=P\K(\ell_2) P$ and 
 $$\J_\alpha=\sigma_1^{-1}[\Psi[\I_\alpha+\K(\ell_2)]]=\sigma^{-1}[P\I_\alpha P+P\K(\ell_2)P]=P\I_\alpha P$$
 for $1\leq\alpha\leq\beta$ as required.
The fact that  $\A$ is generated by $\K(\ell_2)$ and $P\A P$ follows directly from the assumptions.
\end{proof}

\begin{lemma}\label{extend-scattered} Let $\beta>0$ be an ordinal. Suppose that $P\in \B(\ell_2)$ is an infinite rank projection
and a $C^*$-algebra $\A$ satisfying $P\K(\ell_2)P\subseteq \A\subseteq P\B(\ell_2)P$ 
has  an elementary essential composition series $(\I_\alpha)_{\alpha\leq\beta+1}$ 
with $\I_1=P\K(\ell_2) P$ and with $\mathcal T=\{T_{\beta+1, n, m}: m, n\in \N\}$
representing the $(\beta+1)$-th level of $\A$. 
Let $\mathcal S=\{S_{n, m}: m, n\in \N\}\subseteq P^\perp\B(\ell_2)P^\perp$ be a system
of noncompact matrix units and define
$$R_{m,n}=T_{\beta+1, m,n}+S_{m,n}$$
for each $m, n\in \N$. 
 
Then the $C^*$-algebra $\B\subseteq \B(\ell_2)$ generated by $\I_\beta$, $\K(\ell_2)$ and
$\{R_{m,n}: n, m\in \N\}$  has an elementary essential composition series $(\J_\alpha)_{\alpha\leq\beta+1}$ 
such that $\J_1=\K(\ell_2)$, $\J_\alpha=\I_\alpha+\K(\ell_2)$ for $1\leq\alpha<\beta+1$ and
  $\mathcal R=\{R_{ m,n}: m, n\in \N\}$
represents the $(\beta+1)$-th level of $\B$.  
\end{lemma}

\begin{proof}
Let $\Phi: C^*(\mathcal T)\rightarrow C^*(\mathcal S)$ be the isomorphism
such that $\Phi(T_{\beta+1, m, n})=S_{m, n}$. As $\mathcal T$ represent $\beta+1$-th
level of $\A$, each element of $\B$ is 
of the form $A+\Phi(A)+B+K$, where $A\in C^*(\mathcal T)$, $B\in \I_\beta$ and $K\in \K(\ell_2)$.
It follows that $\Psi: \B/\K(\ell_2)\rightarrow \A/P\K(\ell_2)P$ given by 
$$\Psi(A+\Phi(A)+B+\K(\ell_2))=A+B+P\K(\ell_2)P$$
is a well defined isomorphism such that $\Psi(R_{m, n}+\K(\ell_2))=T_{\beta+1, m, n}+P\K(\ell_2)P$ for
each $n, m\in \N$ and $\Psi[\I_\alpha+\K(\ell_2)]=\I_\alpha+P\K(\ell_2)P$ for
$\alpha\leq\beta$. So the properties of $\B$ follow from the properties of $\A$.

%As $\K(\ell_2)\cap \I_\beta=P\K(\ell_2)P$, by Remark 3.1.3 of \cite{murphy} the algebras
%$(\I_\beta+\K(\ell_2))/\K(\ell_2)$ and $\I_\beta/P\K(\ell_2) P$ are canonically $*$-isomorphic. So $\I_\beta+\K(\ell_2)$ 
% has an elementary essential composition series $(\J_\alpha)_{\alpha\leq\beta}$ 
%such that $\J_1=\K(\ell_2)$, $\J_\alpha=\I_\alpha+K(\ell_2)$ for $1\leq\alpha\leq\beta$.
%It is easy to check that $\{R_{m,n}: n, m\in \N\}$ is also a system of matrix units. 
%By the definition of $\B$, the algebra $\B/\J_\beta$ is generated by images of
%$\{R_{m,n}: n, m\in \N\}$ in the quotient homomorphism. Since
%$\J_\beta\cap P^\perp\B(\ell_2) P^\perp\subseteq \K(\ell_2)$ and $S_{m,n}$s 
%are noncompact,    the images of
%$\{R_{m,n}: n, m\in \N\}$ in the quotient homomorphism are matrix units
%in  $\B/\J_\beta$ which witness the fact that  $\B/\J_\beta$
% is isomorphic to $\K(\ell_2)$.
 %So we are left with proving that $\J_{\alpha+1}/\J_\alpha$ for $\alpha<\beta$ 
% are essential ideals of $\B/\J_\alpha$.  For $\alpha=0$ it follows from
% the fact that $\J_0=\K(\ell_2)$ is essential in $\B(\ell_2)$. So suppose that $1\leq \alpha<\beta$.
%  $\Psi:\A/\K(\ell_2)\rightarrow P\A P/P\K(\ell_2)P$  given by
% $\Psi(\pi_1(A))=\sigma_1(PA P)$ for $A\in\A$ is well defined $*$-isomorphism  since e.g.,  $PABP-PAPBP\in \K(\ell_2)$
% for any $A, B\in\A$.
% As in the proof of Lemma \ref{corner-composition} we note that
% $\A/P\K(\ell_2) P$ is isomorphic with $\B/\K(\ell_2)$, so the essentiality follows from
% the essentiality of the ideals $\I_\alpha$ of  $\A$.
  \end{proof}

\begin{theorem}\label{pairing} 
Suppose that $\mathcal S=\{S_{\eta, \xi}: (\xi, \eta)\in \Lambda\}$ is a blockwise system of almost 
matrix units which is separated by a system of projections $(P_\alpha: \alpha<\omega_1)$. Then there is 
 $C^*$-algebra $\mathcal A\subseteq \B(\ell_2)$ with an elementary essential
  composition series $(\I_\alpha)_{\alpha\leq\omega_1}$
with a representing sequence $\mathcal T=(T_{\alpha+1, m, n}: n, m\in \N, \alpha<\omega_1)$
such that
 $\mathcal T$ dominates $\mathcal S$.
\end{theorem}
\begin{proof} 
By recursion on $\beta<\omega_1$ we build  operators $(T_{\alpha+1, m, n}: n, m\in \N, \alpha<\beta)$ and 
stable $C^*$-algebras $\I_\beta\subseteq \B(\ell_2)$  with  an elementary essential
  composition series $(\I_\alpha)_{\alpha\leq\beta}$,
such that $(T_{\alpha+1, m, n}: n, m\in \N, \alpha<\beta)$
 forms  a representing sequence  for   $\I_\beta$ satisfying $(*)$ from Definition \ref{def-domination}. Then 
 $\A = \bigcup_{\alpha<\omega_1} \I_\alpha$ is the required algebra.
 
Let $(T_{1, m, n}: n, m\in \N)$  be the standard system of matrix units of $\K(\ell_2)$
and let $\I_1=\K(\ell_2)$. If $\beta$ is a countable limit ordinal, then $\I_\beta$
is the closure of the union of $\I_\alpha$s for $\alpha<\beta$.

Suppose  now that $(T_{\alpha+1, m, n}: n, m\in \N, \alpha<\beta)$
 and $\I_\beta$ are constructed for $\beta<\omega_1$ and the condition of  Definition \ref{def-domination} is satisfied, i.e., 
for every $0<\alpha<\beta$ we have
$$T_{\alpha+1, m, n}= P_\alpha T_{\alpha+1, m, n} P_\alpha+S_{\lambda_\alpha+m, \lambda_\alpha+n} \leqno(*)$$
for each $n, m\in \N$. We construct $\I_{\beta+1}$ and $(T_{\beta+1, m,n}: n,m\in \N)$ as follows.
 Definition \ref{blockwise-def} (3)-(5) and the above inductive
requirement $(*)$  for every $\alpha<\beta$,   imply that $A-P_\beta AP_\beta\in\K(\ell_2)$
for every $A\in \I_\beta$.
 It follows from Lemma \ref{corner-composition} that $P_\beta\I_\beta P_\beta$ is a 
$C^*$-algebra with
an essential elementary composition series 
$(P_\beta\I_\alpha P_\beta)_{\alpha\leq\beta}$ and  $\I_\beta$ is generated by $P_\beta\I_\beta P_\beta$ and $\K(\ell_2)$.
Since $\K(\ell_2)=\I_1\subseteq \I_\beta$, for each $\alpha\leq\beta$ we have that
$$\I_\alpha=P_\beta\I_\alpha P_\beta+\K(\ell_2).\leqno (**)$$
As observed after (1) -(4) of the Introduction the separable ideals of an elementary essential composition
series are stable and so is $P_\beta\I_\beta P_\beta$.
Now working with $P_\beta\I_\beta P_\beta$  inside $P_\beta\B(\ell_2) P_\beta$
apply Lemma \ref{stable-extension} to obtain  a  stable  $C^*$-algebra $\I_{\beta+1}'\subseteq 
P_\beta\B(\ell_2) P_\beta$ 
with an elementary essential composition series
$(\J_\alpha)_{\alpha\leq\beta+1}$ satisfying $\J_\alpha=P_\beta\I_\alpha P_\beta$ for $\alpha\leq\beta$
and  with a sequence $\{T_{ m, n}': n, m\in \N\}$
representing the $\beta+1$-th level of $\J_{\beta+1}$.
Now for each $n, m\in\N$ define
$$T_{\beta+1, m, n}= T'_{ m, n} +S_{\lambda_\beta+m, \lambda_\beta+n},$$
so $(*)$ of Definition \ref{def-domination} is satisfied.
By Definition \ref{blockwise-def} (5) we have 
$P_\beta^\perp S_{\lambda_\beta+m, \lambda_\beta+n}P_\beta^\perp=
S_{\lambda_\beta+m, \lambda_\beta+n}$, so we are in the position of applying Lemma 
\ref{extend-scattered} to obtain the required $\I_{\beta+1}$. Note that
by $(**)$ and Lemma \ref{extend-scattered} the essential elementary composition series
of $\I_{\beta+1}$ agrees with that of $\I_\beta$, in particular, its $\beta$-th element
is $\I_\beta$ and its representing sequence is $(T_{\alpha+1, m, n}: n, m\in \N, \alpha<\beta+1)$.
This finishes the construction and the proof.
\end{proof}

\section{The nonstability}
In the previous section we showed that for any  blockwise system of
almost matrix units $\mathcal S$ of size $\omega_1$ which is separated by a system of projections,
 we can find a   $C^*$-algebra $\A$ with an essential elementary composition series
 of length $\omega_1$   with a representing sequence
which dominates $\mathcal S$. In this section we will show that if $\mathcal S$ above is a Luzin blockwise system of
almost matrix units, then such $C^*$-algebra $\A$ is not stable. 

\begin{lemma}\label{orthogonal-sequence} Suppose that $\mathcal A$ is a stable  
$C^*$-algebra with an essential elementary composition series $(\I_\alpha)_{\alpha\leq\omega_1}$. Then there 
is a sequence of projections  $(R_\alpha, Q_\alpha: \alpha<\omega_1)\subseteq
\mathcal A$ such that $ht(R_\alpha)=ht(Q_\alpha)=\alpha+1$ 
  for every $\alpha<\omega_1$ and 
we have
$$R_{\alpha_1}Q_{\alpha_2}=0$$
for every $\alpha_1, \alpha_2<\omega_1$.
\end{lemma}
\begin{proof}
Since $\A$ is stable we have  $\mathcal A\cong \A\otimes\K(\ell_2)$.
As $\I_{\alpha+1}/\I_\alpha\cong\K(\ell_2)$, there are projections in $\I_{\alpha+1}/\I_\alpha$, so
let  $R_\alpha''\in \I_{\alpha+1}$ be 
a lifting of such a  projection. We can choose this lifting to be
a projection since each $\I_\alpha$ is a separable AF-ideal (see 
Lemma III. 6. 1. of \cite{davidson}).
%It is clear that $ht(R_\alpha'')=\alpha+1$.
Let $R, Q$ be two orthogonal projections in $\K(\ell_2)$. Put
$R_\alpha'=R_\alpha''\otimes R$ and $Q_\alpha'=R_\alpha''\otimes Q$ for every $\alpha<\omega_1$. Clearly
$R_{\alpha_1}'Q_{\alpha_2}'=0$ for every $\alpha_1, \alpha_2<\omega_1$.
Also $ht(R_\alpha')=ht(Q_\alpha')=\alpha+1$, by Proposition 5.3 of \cite{cb} and Lemma \ref{series-scattered}. 
So $R_\alpha$ and $Q_\alpha$ obtained from $R_\alpha'$ and $Q_\alpha'$ via the isomorphism between
$\A$ and $\A\otimes \K(\ell_2)$,
 satisfy the lemma.
\end{proof}

\begin{theorem}\label{main} Suppose that $\A$ is  $C^*$-algebra
with an essential elementary composition series $(\I_\alpha)_{\alpha\leq\omega_1}$ 
with a representing sequence that  dominates a Luzin blockwise system of
almost matrix units. Then $\A$ is not stable.
\end{theorem}
\begin{proof} 
Let   $\mathcal T_\A=(T_{\alpha+1, m, n}: n, m\in \N, \alpha<\omega_1)\subseteq B(\ell_2)$  be
a representing sequence of $\A$ (Definition \ref{def-representing}) which
 dominates (Definition \ref{def-domination})  a Luzin blockwise  system of
almost matrix units $\mathcal S=\{S_{\eta, \xi}: (\xi, \eta)\in \Lambda\}$  which is separated by a family of projection $(P_\alpha: \alpha<\omega_1)$
(Definition \ref{blockwise-def}).
We will derive
 a contradiction from the hypothesis that $\A$ is stable.

By Lemma \ref{orthogonal-sequence}  there 
is a sequence $(R_\alpha, Q_\alpha: \alpha<\omega_1)\subseteq
\mathcal A$ of projections such that $ht(R_\alpha)=ht( Q_\alpha )=\alpha+1$ 
 for every $\alpha<\omega_1$ and 
 $R_{\alpha}Q_{\beta}=0$
for every $\alpha<\beta<\omega_1$. 

By Lemma \ref{lemma-domination} we have
$$R_\alpha=A_\alpha+ P_\alpha R_\alpha  P_\alpha+F_\alpha, \ \  Q_\alpha =B_\alpha+  P_\alpha  Q_\alpha   P_\alpha+G_\alpha,$$
where $F_\alpha, G_\alpha$ are compact operators and
$A_\alpha, B_\alpha\in  C^*(S_{\alpha})$.  Note that
$\|A_\alpha\|, \|B_\alpha\|>0$ for all $0<\alpha<\omega_1$ as $ht(R_\alpha)=ht( Q_\alpha )=\alpha+1$.

  So by passing
to an uncountable subset $X\subseteq\omega_1$, we may assume that $M>\|A_\alpha\|, \|B_\alpha\|>2\sqrt{\varepsilon}$
and $M>\|F_\alpha\|, \|G_\alpha\|$
for each $\alpha\in X$ and some $M>1$ and $0<\varepsilon<1$. Passing  further down to an  uncountable subset of $X$ 
and using the separability of $\K(\ell_2)$ we may assume
that there are compact operators  $F, G$ such that $\|P_\alpha^\perp F_\alpha-F\|, \|P_\alpha^\perp G_\alpha-G\|<\varepsilon/2M$ 
and $\|F\|, \|G\|<M$
for every $\alpha\in X$. For every $\alpha, \beta\in X$ we have
\begin{align*}
0&=\|R_\alpha Q_\beta\|\geq \|P_\alpha^\perp R_\alpha Q_\beta P_\beta^\perp\|\\
& = \|P_\alpha^\perp(A_\alpha+F_\alpha)(B_\beta+ G_\beta)P_\beta^\perp\|
=\|(A_\alpha+P_\alpha^\perp F_\alpha)(B_\beta+ G_\beta P_\beta^\perp)\|  \\
&\geq\|(A_\alpha+F)(B_\beta+G)\| -\|(P_\alpha^\perp F_\alpha-F)(B_\beta+G)\|  \\ 
 &- \|(A_\alpha+F)(G_\beta P_\beta^\perp -G)\| -\|(P_\alpha^\perp F_\alpha-F)(G_\beta P_\beta^\perp -G)\| \\
&\geq \|(A_\alpha+F)(B_\beta+G)\|
-2\varepsilon-{\varepsilon^2\over M^2}\\
&>\|(A_\alpha+F)(B_\beta+G)\|-3\varepsilon.
\end{align*}

Therefore $\|(A_\alpha+F)(B_\beta+G)\|<3\varepsilon$ for every $\alpha, \beta\in X$.
However, by the Luzin property (Definition \ref{blockwise-luzin-def}) we can  find $\alpha, \beta\in X$ such that
$$\|(A_\alpha+F)(B_\beta+G)\|\geq\|A_\alpha\|\|B_\beta\|-\varepsilon>(2\sqrt{\varepsilon})^2-\varepsilon=3\varepsilon$$
which gives the required contradiction.
Therefore $\A$ is not stable.
\end{proof}

\bibliographystyle{amsplain}

\end{document}